\documentclass[12pt]{article}

\usepackage[utf8]{inputenc}
\usepackage{amsmath}
\usepackage{amsfonts}
\usepackage{amssymb}
\usepackage{mathtools}

\usepackage{cite}

\usepackage{graphicx}
\usepackage{caption}
\usepackage{listings}
\usepackage{verbatim}
\usepackage{wrapfig}
\usepackage{array}
\usepackage{subfigure}

\usepackage{tabularx}

\usepackage{caption}

\usepackage{amsthm}

\theoremstyle{lema}

\theoremstyle{proposition}

\theoremstyle{theorem}
\newtheorem{theorem}{Theorem}[section]

\theoremstyle{theorem}
\newtheorem{remark}{Remark}[section]

\theoremstyle{corollary}

\theoremstyle{definition}
\newtheorem{definition}{Definition}[section]

\theoremstyle{example}
\newtheorem{example}{Example}[section]

\providecommand{\keywords}[1]
{
	\small	
	\textbf{\textit{Keywords---}} #1
}

\providecommand{\msc}[1]
{
	\small	
	\textbf{\textit{Mathematics Subject Classification---}} #1
}

\title{Fixed points for three point generalized orbital triangular contractions}
\author{Cristina Maria Păcurar, Ovidiu Popescu}
\date{}

\begin{document}
	
	\maketitle

	\begin{abstract}
	In this paper we introduce and study new classes of mappings in metric spaces. The main class of mappings is called generalized orbital triangular contractions and it generalizes some existing results (such as Banach contractions, mappings contracting perimeters of triangles). We prove that these contractions are not necessarily continuous and have a unique fixed point under certain conditions. Moreover, we extend our class to generalized orbital triangular Kannan contractions and generalized orbital triangular Chatterjea contractions.
	\end{abstract}
	
	\keywords{Metric space, Banach contraction, generalized contractions, fixed point theorem, Kannan contraction, Chatterjea contraction}
	
	\msc{47H10, 54H25}

	\section{Introduction}

\noindent

The research on fixed point theorems has been a central theme in mathematical analysis from the moment Banach introduced the fixed point result known as the contraction principle (see \cite{Banach}). The continuous interest towards the subject is due to the numerous applications in various fields such as differential equations, dynamical systems, or the theory of computation. 

In light of the joint effort to obtain more general classes of mappings which have fixed points, very recently, Petrov introduced in \cite{Petrov} a new type of mappings in complete metric spaces, which can be characterized as mappings contracting perimeters of triangles. He showed that mappings contracting perimeters of triangles are continuous and proved the fixed point theorem for such mappings. His results were generalized in \cite{Popescu-phi} for mappings contracting (a feature) of triangles, not necessarily perimeters. Also, related results to this new concept were obtained in \cite{Petrov-Chat,Petrov-Kannan,Petrov-pairwise,Popescu-Chat}.

In this paper, we extend the classical frameworks of Banach contractions by introducing a new concept of generalized orbital triangular contractions in metric spaces. As opposed to traditional approaches that often require continuity conditions, our novel class of mappings are not necessarily continuous. For the class of generalized orbital triangular contractions, we establish the existence of unique fixed points. Moreover, we extend our results to Kannan orbital triangular contractions, which are a generalization of Kannan contractions, and to Chatterjea orbital triangular contractions, which are a generalization of Kannan contractions.

\section{Preliminaries}

\noindent

Let us recall some definitions of different types of contractions:

\begin{definition}
	Let $(X,d)$ be a metric space. A mapping $T:X \to X$ is called:
	\begin{itemize}
		\item[i)] Banach contraction (see \cite{Banach}) if there exists $C \in [0,1)$ such that 
		\begin{equation}d(Tx,Ty) \leq C d(x,y),\label{Banach}\end{equation}
		for all $x,\, y \in X$;
		\item[ii)] Kannan contraction (see \cite{Kannan}) if there exists $C \in \left[0,\frac12\right)$ such that 
		\begin{equation} 
		d(Tx,Ty) \leq C [d(x,Tx) + d(y,Ty)],
		\label{Kannan}
		\end{equation}
		for all $x,\, y \in X$;
		\item[iii)] Chatterjea contraction (see \cite{Chatterjea}) if there exists $C \in \left[0,\frac12\right)$ such that 
		\begin{equation}
		d(Tx,Ty) \leq C [d(x,Ty)+d(y,Tx)],
		\label{Chatterjea}
		\end{equation}
		for all $x,\, y \in X$;
		\item[iv)] mapping contracting perimeter of triangles on $X$ (see \cite{Petrov}) if there exists $C \in [0,1)$ such that the inequality 
		\begin{equation}d(Tx,Ty) + d(Ty,Tz) + d(Tz,Tx) \leq C [d(x,y)+d(y,z)+d(z,x)],\label{triangle}\end{equation}
		holds for all three pairwise distinct points $x,y,x \in X$;
		%\item[v)] 
	\end{itemize}
\end{definition}

\begin{definition}
	Let $(X,d)$ be a metric space and $T:X\to X$. A point $x \in X$ is called periodic point of period $n$ if $T^nx =x$. The least positive integer $n$ for which $T^n x=x$ is called the prime period of $x$.
\end{definition}

\section{Generalized orbital triangular contractions}

\noindent

In this section we give the definition of generalized orbital triangular contraction, prove that they are not necessarily continuous and give a fixed point result for such mappings.

\begin{definition}
	Let $(X,d)$ be a metric space. We shall say that $T:X\to X$ is a generalized orbital triangular contraction on $X$ if there exists $\alpha \in [0,1)$ such that the inequality 
	\begin{equation}
	d(Tx,T^2x)+d(T^2x,Ty)+d(Ty,Tx) \leq \alpha [d(x,Tx)+d(Tx,y)+d(y,x)],
	\label{orbital}
	\end{equation}
	holds for all $x,y \in X$, such that $x\neq y \neq Tx$.
	\label{defOrbital}
\end{definition}

%\begin{remark}
%	Every mapping contracting perimeters of triangles is a generalized orbital triangular contraction.

%	Indeed, if $T$ is a mapping contracting perimeters of triangles, then for $x,y,Tx$ pairwise distinct we have 
%\end{remark}

\begin{remark}
	Every Banach contraction is a generalized orbital triangular contraction.
	
	Indeed, if $T$ is a Banach contraction, then for $x\neq y\neq Tx$, by (\ref{Banach}) we have:
	\begin{equation}
	d(Tx,T^2x) \leq C \cdot d(x,Tx),
	\label{Ban-xT}
	\end{equation}
	\begin{equation}
	d(T^2x,Ty) \leq C \cdot d(Tx,y).
	\label{Ban-TT}
	\end{equation}
	\begin{equation}
	d(Ty,Tx) \leq C \cdot d(y,x),
	\label{Ban-xy}
	\end{equation}
	where $C \in [0,1)$.
	
	Adding inequalities (\ref{Ban-xT}), (\ref{Ban-TT}) and (\ref{Ban-xy}) we obtain (\ref{orbital}), so the conclusion follows.
\end{remark}

\begin{remark}
	If we impose the additional condition that $x\neq Tx$ in Definition \ref{defOrbital}, then we obtain that every mapping contracting perimeters of triangles is a generalized orbital triangular contraction.
	
	Indeed, if $T$ is a mapping contracting perimeters of triangles, then for $x,Tx,y$ pairwise distinct, via (\ref{triangle}), we obtain (\ref{orbital}).
	%$$d(Tx,T^2x) + d(T^2x,Ty) + d(Ty,Tx) \leq C [d(x,Tx)+d(Tx,y)+d(y,x)].$$
	
\end{remark}

\begin{theorem}
	Let $(X,d)$ be a complete metric space and let $T:X\to X$ be a generalized orbital triangular contraction on $X$ such that $T$ has no periodic points of prime period $2$. Then, $T$ has a unique fixed point.
	\label{TheoremOrbital}
\end{theorem}

\begin{proof}
	Let $x_0 \in X$, arbitrarily chosen, and $$x_{i+1}=Tx_i$$ for every $i \geq 0$, the Picard iteration. 
	
	Now, suppose that $x_i$ is not a fixed point of $T$, then $x_i\neq x_{i+1} = Tx_i$. Moreover, since $T$ does not have periodic points of prime period $2$, we have $x_{i+2} = TTx_i\neq x_i$ for every $i=0,1,\dots$ 
	
	Then we have $x_{i-1} \neq x_{i+1} \neq x_i = Tx_{i-1}$, and taking $x=x_{i-1}$ and $y=x_{i+1}$ in (\ref{orbital}), we obtain
	\begin{equation*}
	\begin{gathered}
	p_i = d(x_i,x_{i+1})+d(x_{i+1},x_{i+2})+d(x_{i+2},x_{i}) =\\
	=d(Tx_{i-1},T^2x_{i-1})+d(T^2x_{i-1},Tx_{i+1})+d(Tx_{i+1},Tx_{i-1}) \leq \\
	\leq \alpha(d(x_{i-1},Tx_{i-1})+d(Tx_{i-1},x_{i+1})+d(x_{i+1},x_{i-1})) =\\
	=\alpha(d(x_{i-1},x_{i})+d(x_{i},x_{i+1})+d(x_{i+1},x_{i-1})) = \alpha p_{i-1},
	\end{gathered}
	\end{equation*}
	for every $i\geq 1$. So, we obtain 
	$$p_i \leq \alpha p_{i-1} \leq \alpha^2 p_{i-2} \leq \dots \alpha^{i}p_0.$$
	
	Thus, for $m \in \mathbb{N}$, by the triangle inequality we have 
	\begin{equation*}
	\begin{aligned}
	d(x_n,x_{n+m}) &\leq d(x_n,x_{n+1}) + d(x_{n+1},x_{n+2})+ \dots + d(x_{n+m-1},x_{n+m}) \leq \\
	&\leq p_{n+m-2} + \dots +p_{n-1} \leq \\
	&\leq p_0(\alpha^{n-1}+\alpha^{n}+\dots+\alpha^{n+m-2})\\
	& =\alpha^{n-1} \cdot \dfrac{1-\alpha^m}{1-\alpha}p_0\\
	&\leq \dfrac{\alpha^{n-1}}{1-\alpha}p_0,
	\end{aligned}
	\end{equation*}
	so we obtain that $\{x_n\}$ is Cauchy sequence and given the completeness of $X$, we have that $\{x_n\}$ is convergent, so there exists $x^* \in X$ such that $x_n \to x^*$ as $n \to \infty$. Moreover, there exists a subsequence $\{x_{n(k)}\}$ such that $x_{n(k)} \neq x^* \neq Tx_{n(k)}$, so, by (\ref{orbital}) we obtain 
	\begin{equation*}
	\begin{aligned}
	d(Tx_{n(k)},T^2x_{n(k)}) &+ d(T^2x_{n(k)},Tx^*) + d(Tx^*,Tx_{n(k)}) \leq \\ &\leq \alpha [d(x_{n(k)},Tx_{n(k)})+d(Tx_{n(k)},x^*)+d(x^*,x_{n(k)})].
	\end{aligned}
	\end{equation*}
	Hence, we get
	\begin{equation*}
	\begin{aligned}
	d(x_{n(k)+1},x_{n(k)+2}) &+ d(x_{n(k)+1},Tx^*) + d(Tx^*,x_{n(k)+2}) \leq \\ &\leq \alpha [d(x_{n(k)},x_{n(k+1)})+d(x_{n(k+1)},x^*)+d(x^*,x_{n(k)})],
	\end{aligned}
	\end{equation*}
	where taking the limit as $k \to \infty$ we get 
	$$2d(x^*,Tx^*) \leq 0,$$
	by where $d(x^*,Tx^*)=0$, so $x^*$ is a fixed point of $T$.
	
	Now, suppose that there exists another fixed point of $T$, $y^*\in X$ such that $Ty^* = y^*$ and $x^* \neq y^*$. Then $Tx^* \neq y^*$, so, by (\ref{orbital}) we obtain 
	\begin{equation*}
	d(Tx^*,T^2x^*)+d(T^2x^*,Ty^*)+d(Ty^*,Tx^*)  \leq \alpha [d(x^*,Tx^*)+d(Tx^*,y^*)+d(y^*,x^*)],
	\end{equation*}
	so we get $$2d(x^*,y^*) \leq 2\alpha d(x^*,y^*),$$ which is a contradiction since $\alpha \in [0,1)$. Thus, we obtain $x^* =y^*$, so $T$ has a unique fixed point.		
\end{proof}

\begin{remark}
	If we impose the additional condition that $x\neq Tx$ in Definition \ref{defOrbital}, then $T$ can have many fixed points.
\end{remark}

\begin{remark}
	Generalized orbital triangular contractions on $X$ are not necessarily continuous.
	
	Indeed, let $X =[0,1]$ and $T: X \to X$ defined as $$Tx=\begin{cases}
	0, \quad x\in [0,1)\\
	\dfrac14, \quad x =1.
	\end{cases}$$
	
	Obviously, $T$ is not continuous at $x=1$, but $T$ is a generalized orbital triangular contraction for $\alpha = \dfrac23 <1$.
	
	Indeed, if $x,y \in [0,1)$, then $Tx=T^2x=Ty=0$, so
	$$d(Tx,T^2x)+d(T^2x,Ty)+d(Ty,Tx) = 0 \leq \dfrac23 (d(x,Tx)+d(Tx,y)+d(y,x)).$$
	
	If $x=1$, then $Tx=\dfrac14$ and for $y\neq x$, we have $T^2x=Ty=0$, so 
	$$d(Tx,T^2x)+d(T^2x,Ty)+d(Ty,Tx) = \dfrac 14 + 0+ \dfrac 14 = \dfrac 12$$
	and $$d(x,Tx)+d(Tx,y)+d(y,x) = 1-\dfrac14+\left|\dfrac14-y\right|+1-y =$$
	$$= \dfrac74-y+\left|\dfrac14-y\right| \geq \dfrac74-y \geq \dfrac74-1 = \dfrac34,$$
	so for $\alpha = \dfrac 23$, since
	$\dfrac12 \leq \dfrac23 \cdot \dfrac 34,$ (\ref{orbital}) is true.
	
	If $y=1$, then $Ty=\dfrac14$ and since $x \neq y$, $Tx=T^2x=0$ so we have 
	$$d(Tx,T^2x)+d(T^2x,Ty)+d(Ty,Tx) = 0+\dfrac 14 + \dfrac 14 = \dfrac 12$$
	and $$d(x,Tx)+d(Tx,y)+d(y,x) = x+1+1-x =2,$$
	and since
	$\dfrac12 \leq \dfrac23 \cdot 2,$ we obtain that (\ref{orbital}) is true.
	
	So, $T$ is generalized orbital triangular contraction which has a unique fixed point $0=T0$.
\end{remark}

\begin{remark}
	There exist generalized orbital triangular contractions which are not Banach contractions \cite{Banach}, nor mappings contracting perimeters of triangles \cite{Petrov}, neither mappings contracting triangles \cite{Popescu-phi}.
	
	Indeed, since generalized orbital triangular contractions are not necessarily continuous, the conclusion follows.
\end{remark}

\section{Generalized orbital triangular Kannan contractions}

\noindent

In this section, we give the definition of generalized orbital triangular Kannan contractions and prove a fixed point theorem for such mappings.

\begin{definition}
	Let $(X,d)$ be a metric space. We shall say that $T:X\to X$ is a generalized orbital triangular Kannan contraction on $X$ if there exists $\beta \in [0,\frac23)$ such that the inequality 
	\begin{equation}
	d(Tx,T^2x)+d(T^2x,Ty)+d(Ty,Tx) \leq \beta [d(x,Tx)+d(y,Ty)+d(Tx,T^2x)],
	\label{orbitalKannan}
	\end{equation}
	holds for all $x,y \in X$, such that $x, y, Tx$ are pairwise distinct.
\end{definition}

\begin{remark}
	Every Kannan contraction with $C <\dfrac13$ is a generalized orbital triangular Kannan contraction.
	
	Indeed, if $T$ is a Kannan contraction, then for $x,y,Tx$ pairwise distinct, by (\ref{Kannan}) we have:
	\begin{equation}
	d(Tx,T^2x) \leq C [d(x,Tx)+d(Tx,T^2x)],
	\label{Kannan-xT}
	\end{equation}
	\begin{equation}
	d(T^2x,Ty) \leq C [d(Tx,T^2x)+d(y,Ty)].
	\label{Kannan-TT}
	\end{equation}
	\begin{equation}
	d(Ty,Tx) \leq C [d(y,Ty)+d(x,Tx)].
	\label{Kannan-xy}
	\end{equation}
	
	Adding inequalities (\ref{Kannan-xT}), (\ref{Kannan-TT}) and (\ref{Kannan-xy}) we obtain (\ref{orbitalKannan}), so the conclusion follows.
\end{remark}

\begin{theorem}
	Let $(X,d)$ be a complete metric space and let $T:X\to X$ be a generalized orbital triangular Kannan contraction on $X$ such that $T$ has no periodic points of prime period $2$. Then, $T$ has a fixed point.
	\label{TheoremKannan}
\end{theorem}

\begin{proof}
	We suppose that $T$ has no periodic points of prime period $2$. And for $x_0 \in X$ arbitrarily chosen, let the Picard iteration $x_n=Tx_{n-1}$ for all $n \geq 0$. Suppose that $x_n$ is not a fixed point of $T$ for every $n \geq 0$. Then, $x_n,Tx_n=x_{n+1}$ and $T^2x_n=x_{n+2}$ are pairwise distinct, and choosing $x=x_n$, $y=T^2x_n$ in (\ref{orbitalKannan}) we have 
	\begin{equation*}
	\begin{gathered}
	d(Tx_n,T^2x_n)+d(T^2x_n,T^3x_n)+d(T^3x_n,Tx_n) \leq\\ \leq \beta [d(x_n,Tx_n)+d(T^2x_n,T^3x_n)+d(Tx_n,T^2x_n)],
	\end{gathered}
	\end{equation*}
	thus, using the triangle inequality, we get 
	\begin{equation}
	\begin{gathered}
	(1-\beta)[d(Tx_n,T^2x_n)+d(T^2x_n,T^3x_n)]+|d(Tx_n,T^2x_n)-d(T^2x_n,T^3x_n)| \leq \\ \leq
	(1-\beta)[d(Tx_n,T^2x_n)+d(T^2x_n,T^3x_n)]+d(T^3x_n,Tx_n) \leq \\ \leq \beta [d(x_n,Tx_n)].
	\end{gathered}
	\label{Kan1}
	\end{equation}
	
	Now, let $d_n = d(x_n,x_{n+1})$. By (\ref{Kan1}), we obtain
	\begin{equation}
	(1-\beta)(d_{n+1}+d_{n+2})+|d_{n+1}-d_{n+2}| \leq \beta d_n.
	\label{Kan2dn}
	\end{equation}
	
	If $d_{n+1} \geq d_{n+2}$, (\ref{Kan2dn}) becomes
	\begin{equation*}
	(2-\beta)d_{n+1} - \beta d_{n+2} \leq \beta d_n,
	\end{equation*} 
	by where
	\begin{equation*}
	(2-2\beta)d_{n+1} \leq \beta d_n,
	\end{equation*}
	so we obtain 
	\begin{equation}
	d_{n+1}\leq \dfrac{\beta}{2-2\beta}d_n.
	\label{Kan3}
	\end{equation}
	
	If $d_{n+1} < d_{n+2}$, then, (\ref{Kan2dn}) becomes
	\begin{equation*}
	(2-\beta)d_{n+2}\leq \beta (d_n+d_{n+1}),
	\end{equation*}
	so we have
	\begin{equation*}
	(2-\beta)d_{n+1} \leq(2-\beta)d_{n+2}\leq \beta (d_n+d_{n+1}),
	\end{equation*}
	and thus 
	\begin{equation}
	d_{n+1}\leq \dfrac{\beta}{2-2\beta}d_n.
	\label{Kan4}
	\end{equation}
	
	By (\ref{Kan3}) and (\ref{Kan4}), since $\beta < \dfrac23$, $\lambda=\dfrac{\beta}{2-2\beta} <1$ and we have 
	\begin{equation}
	d_{n+1} \leq \lambda d_n.
	\label{KanDn}
	\end{equation}
	
	Thus, for $m \in \mathbb{N}$, by the triangle inequality and (\ref{KanDn}), we have 
	\begin{equation*}
	\begin{aligned}
	d(x_n,x_{n+m}) &\leq d(x_n,x_{n+1}) + d(x_{n+1},x_{n+2})+ \dots + d(x_{n+m-1},x_{n+m}) = \\
	&= d_{n} + \dots +d_{n+m-1} \leq \\
	&\leq \lambda^{n-1}d_0+\lambda^{n}d_0+\dots+\lambda^{n+m-1}d_0\\
	& =\lambda^{n-1} \cdot \dfrac{1-\lambda^m}{1-\lambda}d_0\\
	&\leq \dfrac{\lambda^{n-1}}{1-\lambda}d_0,
	\end{aligned}
	\end{equation*}
	so, since $\lambda <1$, we obtain that $\{x_n\}$ is Cauchy sequence and given the completeness of $X$, we have that $\{x_n\}$ is convergent, so there exists $x^* \in X$ such that $x_n \to x^*$ as $n \to \infty$.
	
	Moreover, there exists a subsequence $\{x_{n(k)}\}$ such that $x_{n(k)} \neq x^* \neq Tx_{n(k)}$, so by (\ref{orbitalKannan}) we have 
	\begin{equation*}
	\begin{gathered}
	d(Tx_{n(k)},T^2x_{n(k)}) + d(T^2x_{n(k)},Tx^*) + d(Tx^*,Tx_{n(k)}) \leq \\  \leq \beta [d(x_{n(k)},Tx_{n(k)})+d(x^*,Tx^*)+d(Tx_{n(k)},T^2x_{n(k)})],
	\end{gathered}
	\end{equation*}
	so
	\begin{equation*}
	\begin{aligned}
	d(x_{n(k)+1},x_{n(k)+2}) &+ d(x_{n(k)+1},Tx^*) + d(x_{n(k)+2},Tx^*) \leq \\ &\leq \beta [d(x_{n(k)},x_{n(k)+1})+d(x^*,Tx^*)+d(x_{n(k)+1},x_{n(k)+2})],
	\end{aligned}
	\end{equation*}
	and taking the limit as $k \to \infty$ we get 
	$$2d(x^*,Tx^*) \leq \beta \cdot d(x^*,T^*),$$
	by where $d(x^*,Tx^*)=0$, so $x^*$ is a fixed point of $T$.
\end{proof}
%\end{comment}

\begin{example}
	Let $X=\{A,B,C,D\}$ and as in Figure \ref{Fig}:
	$$d(A,B)=d(A,C)=d(A,D)=4, \quad d(B,C)=d(C,D)=1, \quad d(B,D)=2,$$ and
	$T:X\to X$ defined as
	$$TA=TC=C, \quad TB=B, \quad TD=D.$$
	
	\begin{figure}[h!]
		\centering
		\includegraphics[width=6cm]{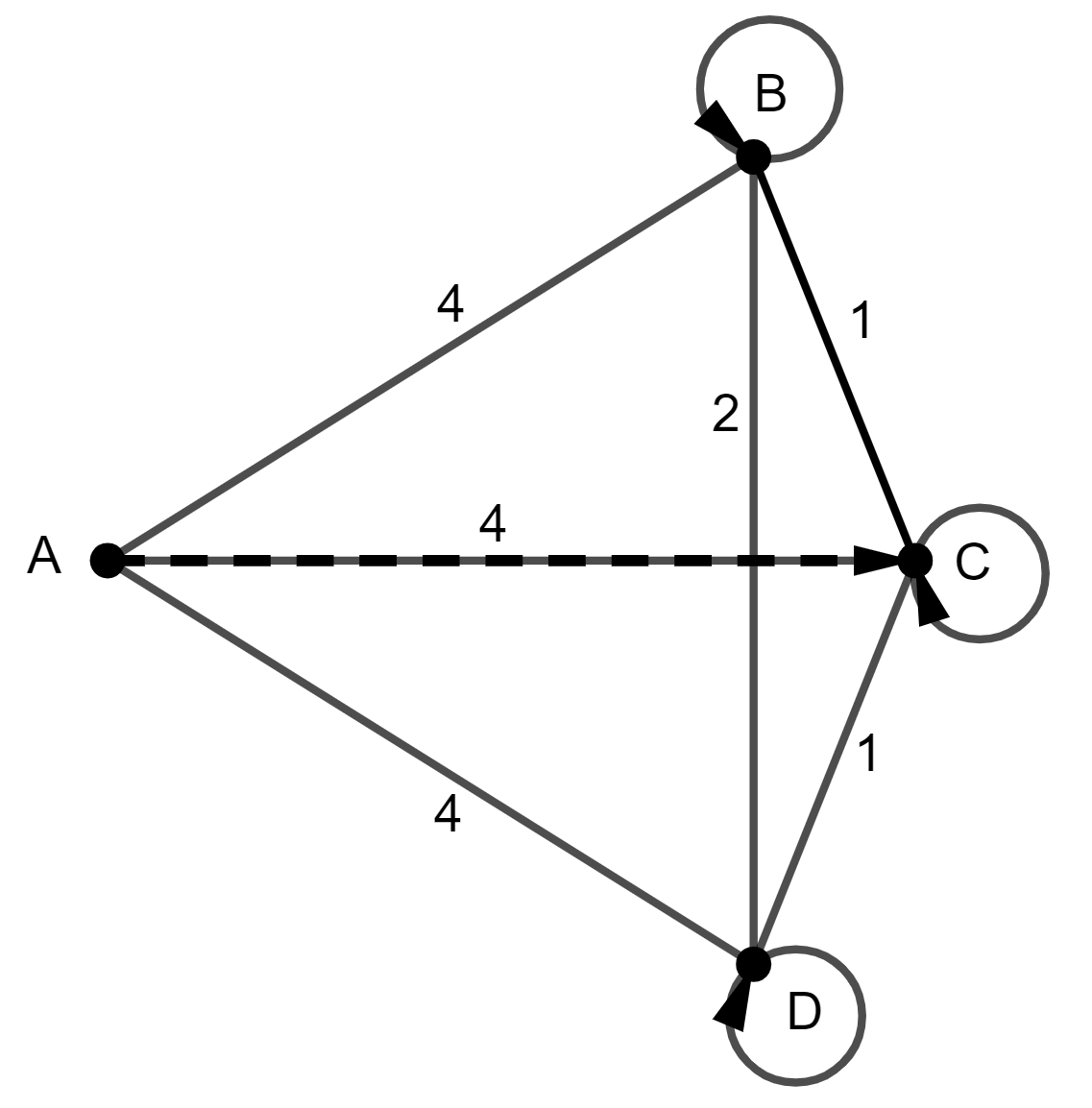}
		\caption{Example of generalized orbital triangular Kannan mapping}
		\label{Fig}
	\end{figure}
	
	Since $d(TB,TD) = 2$ and $d(B,TB)+d(D,TD) =0$, $T$ is not a Kannan contraction.
	
	Moreover, since 
	$$d(TA,TB)+d(TB,TD)+d(TD,TA)= 4 = d(A,TA)+d(B,TB)+d(D,TD),$$
	$T$ is not a generalized Kannan mapping (see \cite{Petrov-Kannan}).
	
	However, we have 
	$$d(TA,T^2A)+d(T^2A,TB)+d(TB,TA)= 2 \leq \dfrac12 \cdot 4 = \dfrac12[d(A,TA)+d(B,TB)+d(TA,T^2A)],$$
	$$d(TA,T^2A)+d(T^2A,TD)+d(TD,TA)= 2 \leq \dfrac12 \cdot 4 = \dfrac12[d(A,TA)+d(D,TD)+d(TA,T^2A)],$$
	so $T$ is a generalized orbital triangular Kannan contraction.
	
	Let us note that $T$ does not have periodic points of prime period $2$, and $T$ is not a Banach contraction either as $d(TB,TD) = 2 = d(B,D)$.
\end{example}

\section{Generalized orbital triangular Chatterjea contractions}

\noindent

In this section, we give the definition of generalized orbital triangular Chatterjea contractions and prove a fixed point theorem for such mappings.

\begin{definition}
	Let $(X,d)$ be a metric space. We shall say that $T:X\to X$ is a generalized orbital triangular Chatterjea contraction on $X$ if there exists $\gamma \in \left[0,\frac12\right)$ such that the inequality 
	\begin{equation}
	\begin{gathered}
	d(Tx,T^2x)+d(T^2x,Ty)+d(Ty,Tx) \leq \\ \leq \gamma [d(x,Ty)+d(y,Tx)+d(x,T^2x)+d(y,T^2x)+d(Tx,Ty)],
	\label{orbitalChatterjea}
	\end{gathered}
	\end{equation}
	holds for all $x,y \in X$, such that $x\neq y \neq Tx$.
\end{definition}

\begin{remark}
	Every Chatterjea contraction is a generalized orbital triangular Chatterjea contraction.
	
	Indeed, if $T$ is a Chatterjea contraction, then there exists $C \in [0,\frac12)$ such that for $x\neq y \neq Tx$ by (\ref{Chatterjea}) we have:
	\begin{equation}
	d(Tx,T^2x) \leq C [d(x,T^2x)+d(Tx,Tx)],
	\label{Chatterjea-xT}
	\end{equation}
	\begin{equation}
	d(T^2x,Ty) \leq C [d(Tx,Ty)+d(y,T^2x)].
	\label{Chatterjea-TT}
	\end{equation}
	\begin{equation}
	d(Ty,Tx) \leq C [d(x,Ty)+d(y,Tx)].
	\label{Chatterjea-xy}
	\end{equation}
	
	Adding inequalities (\ref{Chatterjea-xT}), (\ref{Chatterjea-TT}) and (\ref{Chatterjea-xy}) we obtain (\ref{orbitalChatterjea}), so the conclusion follows.
\end{remark}

\begin{theorem}
	Let $(X,d)$ be a complete metric space and let $T:X\to X$ be a generalized orbital triangular Chatterjea contraction on $X$ such that $T$ has no periodic points of prime period $2$. Then, $T$ has a unique fixed point.
	\label{TheoremChatterjea}
\end{theorem}

\begin{proof}
	Let $x_i$ and the Picard iteration as in the proof of Theorem \ref{TheoremOrbital} and $$p_i = d(x_i,x_{i+1})+d(x_{i+1},x_{i+2})+d(x_{i+2},x_{i}),$$ for every $i \geq 0$. Then, by (\ref{orbitalChatterjea}), we have
	\begin{equation*}
	\begin{gathered}
	p_i = d(x_i,x_{i+1})+d(x_{i+1},x_{i+2})+d(x_{i+2},x_{i}) =\\
	=d(Tx_{i-1},T^2x_{i-1})+d(T^2x_{i-1},Tx_{i+1})+d(Tx_{i+1},Tx_{i-1}) \leq \\
	\leq \gamma[d(x_{i-1},Tx_{i+1})+d(x_{i+1},Tx_{i-1})+d(x_{i-1},T^2x_{i-1})+\\+d(x_{i+1}, T^2x_{i-1})+d(Tx_{i-1},Tx_{i+1})] =\\
	=\gamma[d(x_{i-1},x_{i+2})+d(x_{i+1},x_{i})+d(x_{i-1},x_{i+1})+d(x_i,x_{i+2})] \leq \\
	%\leq \gamma(d(x_{i-1},x_{i})+d(x_{i},x_{i+1})+d(x_{i+1},x_{i}) + d(x_{i},x_{i+2})) \leq\\
	\leq \gamma[d(x_{i-1},x_{i})+d(x_{i},x_{i+1})+d(x_{i+1},x_{i-1}) +\\+ d(x_i,x_{i+1})+d(x_{i+1},x_{i+2})+d(x_{i+2},x_i)]=\\=
	\gamma\cdot(p_{i-1}+p_i),
	\end{gathered}
	\end{equation*}
	for every $i \geq 1$, so we obtain $$p_i \leq \dfrac{\gamma}{1-\gamma} p_{i-1}$$ for every $i \geq 1$. Thus, since $\dfrac{\gamma}{1-\gamma} < 1$, as in the proof of Theorem \ref{TheoremOrbital} we obtain that $\{x_n\}$ is a Cauchy sequence, so it is a convergent sequence to an $x^* \in X$. 
	
	To prove that $x^*$ is a fixed point of $T$, notice that since $\{x_n\}$ is convergent and $x_n \neq x_{n+1} \neq x_{n+2}$, there exists a subsequence $\{x_{n(k)}\}$ such that $x_{n(k)} \neq Tx_{n(k)} \neq x^*$, so, by (\ref{orbitalChatterjea}) we obtain 
	\begin{equation*}
	\begin{gathered}
	d(Tx_{n(k)},T^2x_{n(k)}) + d(T^2x_{n(k)},Tx^*) + d(Tx^*,Tx_{n(k)}) \leq \\  \leq \gamma [d(x_{n(k)},Tx^*)+d(x^*,Tx_{n(k)})+d(x_{n(k)},T^2x_{n(k)})+\\+d(x^*,T^2x_{n(k)})+d(Tx_{n(k)},Tx^*)],
	\end{gathered}
	\end{equation*}
	so
	\begin{equation*}
	\begin{gathered}
	d(x_{n(k)+1},x_{n(k)+2}) + d(x_{n(k)+1},Tx^*) + d(x_{n(k)+2},Tx^*) \leq \\ \leq \gamma [d(x_{n(k)},Tx^*)+d(x^*,x_{n(k+1)})+d(x_{n(k)},x_{n(k+2)})+\\+d(x^*,x_{n(k+2)})+d(x_{n(k)+1},Tx^*)],
	\end{gathered}
	\end{equation*}
	and taking the limit as $k \to \infty$ we get 
	$$2d(x^*,Tx^*) \leq 2\gamma d(x^*,T^*),$$
	by where $d(x^*,Tx^*)=0$, so $x^*$ is a fixed point of $T$.
	
	Let us suppose that there exists another fixed point of $T$, $y^* \in X$ for which $Ty^*=y^*\neq x^* \neq Tx^*$. Then, by (\ref{orbitalChatterjea}) we get 
	\begin{equation*}
	\begin{gathered}
	d(Tx^*,T^2x^*)+d(T^2x^*,Ty^*)+d(Ty^*,Tx^*) \leq \\ \leq \gamma [d(x^*,Ty^*)+d(y^*,Tx^*)+d(x^*,T^2x^*)+d(y^*,T^2x^*)+d(Tx^*,Ty^*)],
	\end{gathered}
	\end{equation*}
	so $$2d(x^*,y^*) \leq 4\gamma d(x^*,y^*),$$ which is a contradiction since $\gamma < \dfrac12$.
\end{proof}

\begin{example}
	Let $X=\{0,1,2,3\}$ endowed with the distance $d(x,y) =|x-y|$ and let $T:X\to X$ defined as $$T0=T1=T2=0, \quad T3=2.$$
	Let us first note that $T$ has a unique fixed point $0=T0$ and $T$ does not possess any periodic points of prime period $2$ since
	$$T^21=T^22=T^23=0.$$
	
	Moreover, $T$ is not a Chatterjea contraction since we have 
	$$d(T2,T3) = |0-2|=2$$ and $$d(2,T3)+d(3,T2) = |2-2|+|3-0|=3.$$
	
	However, for every $x,y\in X$ such that $x\neq y\neq Tx$, letting $$L(x,y) := d(Tx,T^2x)+d(T^2x,Ty)+d(Ty,Tx)$$ and $$R(x,y):=d(x,Ty)+d(y,Tx)+d(x,T^2x)+d(y,T^2x)+d(Tx,Ty)$$ we have 
	$$L(0,1)=L(0,2)=L(1,2)=L(2,1) = 0,$$
	$$L(0,3) = L(3,0) = L(1,3)=L(3,1) = L(2,3)= 4$$
	and 
	$$R(0,1)= 3, \quad R(0,2)= 4, \quad R(1,2)=R(2,1) = 6,$$
	$$R(0,3) = R(3,0) = R(1,3)=R(3,1) = R(2,3) =10. $$
	
	Thus, $T$ is a generalized orbital triangular Chatterjea contraction with $\gamma=\dfrac25$.
\end{example}


	%\hspace{30pt}
	\medskip
	\vspace{1.2ex}
	%\usebox{\authors}
	
\end{document}